%% file: project.tex
\documentclass[12pt]{article}

\usepackage{amsmath}
\usepackage{amssymb}
\usepackage{amsthm}
\usepackage[all,cmtip]{xy}
\usepackage[pdftex]{graphicx}
\usepackage{asymptote}
\usepackage{hyperref}

\newtheorem{theorem}{Theorem}
\newtheorem{lemma}[theorem]{Lemma}
\newtheorem{corollary}[theorem]{Corollary}
\newtheorem{definition}{Definition}
\newtheorem{remark}{Remark}
\newtheorem{example}{Example}
\newtheorem{question}{Question}
\newtheorem{notation}{Notation}

\DeclareMathOperator{\rk}{rank}

\title{Coverings over Tori and Topological Approach to Klein's Resolvent Problem}

\begin{document}

\author{Y. Burda}

\maketitle


\begin{abstract}
This work answers the question what coverings over a topological torus can be induced from a covering over a space of dimension $k$. The answer to this question is then applied in algebro-geometric context to present obstructions to transforming an algebraic equation depending on several parameters to an equation depending on less parameters by means of a rational transformation.
\end{abstract}

\bibliographystyle{acm}

\tableofcontents

\section{Introduction}

The goal of this article is to develop a topological approach to Klein's resolvent problem. This problem asks for the minimal number of independent parameters on which a given algebraic equation depending on several parameters can be made to depend after a rational transformation is applied to it (see section \ref{section:formulations} below for a precise formulation).

The approach is to make precise the statement that a complicated enough monodromy of an algebraic function might prevent it from living on a space of small enough dimension. In \cite{Arn1} Arnold proposed to use for this purpose characteristic classes of algebraic functions with values in cohomology groups of the space on which the function is defined. Another approach, dual to this one in some sense, uses critical submanifolds in the base space instead of characteristic classes. It was first proposed in \cite{Cheb1} (this article contained uncorrectable mistakes) and later developed in \cite{Lin1} to give very strong results.

These approaches work well when the goal is to show that a given algebraic function can't be induced from an algebraic function on a space of low dimension by means of a polynomial mapping. In other words they can be applied to prove that a given algebraic function can't be expressed using a formula involving the operations of addition, subtraction, multiplication and solving one algebraic equation depending on a small number of parameters. These methods fail however when the operation of division is allowed: when one throws away an arbitrary hypersurface from the base space its topology can change in unexpected ways.

To overcome this obstacle Buhler and Reichstein developed in \cite{Buh2} purely algebraic methods to approach Klein's resolvent problem. In \cite{Buh1}, following an approach suggested by Serre, they used algebraic analogues of Stiefel-Whitney classes taking values in Galois cohomology of the base-field of an algebraic extension of fields, making the algebraic approach similar in spirit to the topological approach of Arnold's work \cite{Arn1}.

In this article we present an approach that uses mostly geometric and topological methods.  More precisely we use a certain family of tori in the base-space of an algebraic function with the properties that the restrictions of the algebraic function to all these tori are topologically equivalent and for any hypersurface one can find a torus in this family that lives in the compliment to this hypersurface. It turns out that sometimes, using characteristic classes for coverings, one can show that the restriction of the algebraic function to each of these tori is ``complicated enough'' so that it can't be induced from any algebraic function on a variety of a small dimension.

To make this plan work we solve completely in sections \ref{section:notations}-\ref{section:tori} the problem of determining whether a given covering over a topological torus can be induced from a covering over a topological space of dimension $k$. The answer turns out to be ``for $k$ greater than or equal to the rank of the monodromy group of the covering''.

In the second part of this article (sections \ref{section:formulations}-\ref{section:generic}) we apply these topological results in the context of Klein's resolvent problem. In section \ref{section:algebraic_tori} we completely solve Klein's resolvent problem for algebraic functions unramified over the algebraic torus and use this result to prove some estimates in Klein's resolvent problem for the universal algebraic function. In section \ref{section:local_version} we give a more general construction which can be applied to get estimates in Klein's resolvent problem for any algebraic function and apply them to reprove the estimates for the universal algebraic function. Finally in \ref{section:generic} we apply these methods to show that generically one should expect that an algebraic function of degree $\geq 2k$ depending on $k$ parameters doesn't admit any rational transformation that makes it depend on a smaller number of parameters.

The author would like to express his gratitude to his advisor A.G. Khovanskii for warm support for this work and fruitful discussions about Klein's resolvent problem. The author also thanks M. Mazin for explaining how the notions of a Parshin point and its neighbourhood can be applied in geometric context.

\section{Notations}
\label{section:notations}

\subsection{Spaces and their dimensions}

We will be dealing in this work with coverings over topological spaces. When we say a \textit{space} what we will mean a topological space which admits a universal covering and is homotopically equivalent to a CW-complex. We will say that a space is \textit{of dimension at most $k$} if it is homotopically equivalent to a CW-complex with cells of dimension at most $k$. A \textit{mapping} between two spaces will refer to a continuous mapping.

Any constructible algebraic set over the complex numbers is a space in the above meaning. Moreover, an affine variety of complex dimension $k$ is a space of dimension at most $k$. 

\subsection{Coverings and their monodromy}
\label{section:monodromy}

The notation $\xi_X$ will denote a covering $p_X:(\tilde{X},\tilde{x}_0)\to(X,x_0)$ between pointed spaces $(\tilde{X},\tilde{x}_0)$ and $(X,x_0)$. The space $X$ will be assumed to be connected, but $\tilde{X}$ --- not necessarily so.

For a given covering $\xi_X$ we can consider the \textit{monodromy representation} of the fundamental group $\pi_1(X,x_0)$ on the fiber of $p_X$ over the basepoint $x_0$. Namely the monodromy representation is a group homomorphism $M_X:\pi_1(X,x_0)\rightarrow S(p_X^{-1}(x_0))$ which maps the class of a loop $\gamma$ in the fundamental group to the permutation that sends the point $\tilde{x}\in p_X^{-1}(x_0)$ to the other endpoint of the unique lift $\tilde{\gamma}$ of the loop $\gamma$ to a path in $\tilde{X}$ starting at $\tilde{x}$. The image of the monodromy representation is called the \textit{monodromy group}.

In fact specifying a covering over $(X,x_0)$ is equivalent to specifying the fiber over the basepoint $x_0$, a point in this fiber, and the monodromy action of $\pi_1(X,x_0)$ on the fiber. Indeed, given a set $L$, an action $M:\pi_1(X,x_0)\rightarrow S(L)$ of the fundamental group $\pi_1(X,x_0)$ on $L$ and a point $l_0\in L$, we can construct a covering over $(X,x_0)$, whose fiber over $x_0$ can be identified with $L$ and with this identification the basepoint of the total space of the covering gets identified with $l_0$ and the monodromy representation of the fundamental group gets identified with $M$.

To do so let $p_X^u:(U,u_0)\rightarrow (X,x_0)$ denote the universal covering over $(X,x_0)$. The fundamental group $\pi_1(X,x_0)$ acts on the total space $U$ via deck transformations of the universal covering. We define $\tilde{X}$ as the quotient of the product space of $U\times L$ by the equivalence relation $(\alpha\cdot u,l) \sim (u,M(\alpha)\cdot l)$, where $u\in U$, $l\in L$ and $\alpha\in \pi_1(X,x_0)$. Let $[u,l]$ denote the class of equivalence of point $(u,l)\in U\times L$ under this equivalence. We choose the point $\tilde{x}_0=[u_0,l_0]$ as the basepoint of $\tilde{X}$. The covering map $p_X:(\tilde{X},\tilde{x}_0)\to (X,x_0)$ is defined by the formula $p_X([u,l])=p_X^u(u)$ (it is easy to see that this map is well-defined and is a covering map). Since the action of $\pi_1(X,x_0)$ on the fiber of the universal covering over $x_0$ is free and transitive, each point $[u,l]\in p_X^{-1}(x_0)$ is represented by a unique pair of the form $(u_0,l')$. We will identify this point with the element $l'\in L$. With this identification the monodromy action of the constructed covering on the fiber $p_X^{-1}(x_0)$ is identified with $M$ and the basepoint $[u_0,l_0]$ gets identified with $l_0$.

\section{\texorpdfstring{$G$}{G}-labelled coverings}

Let $G$ be a group. A \textit{$G$-labelled covering} $\xi_X$ is a covering map $p_X:(\tilde{X},\tilde{x}_0)\to (X,x_0)$ between pointed topological spaces together with an identification of the monodromy group with a subgroup of $G$. Explicitly, the labelling is an injective group homomorphism $L_X:M_X(\pi_1(X,x_0))\to G$ from the monodromy group to $G$. We will refer to the composition of the monodromy representation $M_X$ and the labelling map $L_X$ simply as ``the monodromy map'' $\mathcal{M}_X=L_X\circ M_X:\pi_1(X,x_0)\to G$. The image of $\mathcal{M}_X$ in $G$ will be called ``the monodromy group of the $G$-labelled covering'', or, if no confusion could arise, ``the monodromy group''.

For every map $f:X\to Y$ and any $G$-labelled covering $\xi_Y$ we can define the induced $G$-labelled covering $f^*\xi_Y$ over $X$ in the obvious way. For a map $f:X\to Y$ which induces a surjective homomorphism on the fundamental groups, the monodromy group of $f^*\xi_Y$ is equal to that of $\xi_Y$.

The main objects of inquiry in this part of the work will be coverings, however the additional structure of labelling has to be introduced for the definition of characteristic classes below: the value of a characteristic class on a covering with monodromy group isomorphic to $G$ may depend on the labelling.

Note also that any covering can be considered as an $S(n)$-labelled covering, once the fiber over the base-point is identified with the set of labels $1,\ldots,n$.

\section{Characteristic classes for coverings}\label{sectionchar}

In the definition below we will use the category of coverings whose objects are coverings over connected topological spaces and morphisms between two coverings $\xi_X$ and $\xi_Y$ are pairs of maps $(f,g)$ making the diagram
$$\xymatrix{
(\tilde{X},\tilde{x}_0) \ar[d]^{p_X} \ar[r]^{g} & (\tilde{Y},\tilde{y}_0) \ar[d]^{p_Y} \\
(X,x_0)  \ar[r]^{f} & (Y,y_0) 
}
$$
commutative and such that $g:p_X^{-1}(x)\to p_Y^{-1}(f(x))$ is a bijection for every  $x\in X$.

The category of $G$-labelled coverings has $G$-labelled coverings as objects and a morphism of $G$-labelled coverings $\xi_X$ and $\xi_Y$ is a morphism of the underlying coverings with the additional requirement about the labellings: $\mathcal{M}_X=\mathcal{M}_Y\circ f_*$.

\begin{definition}
Let $\cal{C}$ be any subcategory of the category of coverings or of the category of $G$-labelled coverings for some group $G$. A \textbf{characteristic class} for category $\mathcal{C}$ of degree $k$ with coefficients in an abelian group $A$ is a mapping $w$ which assigns to any covering $\xi_X$ from $\cal{C}$ a cohomology class $w(\xi_X)\in H^k(X,A)$ such that if $(f,g)$ is a morphism from $\xi_X$ to $\xi_Y$ then  $w(\xi_Y)=f^*w(\xi_X)$.
\end{definition}

Here are some examples of characteristic classes.

\begin{example}
\label{example:classifying}
Let $G$ be a discrete group and $A$ --- an abelian group. Let $(BG,b_0)$ be the classifying space for the group $G$. For $k>0$ let $w\in H^k(BG,A)$ be any class in the cohomology of group $G$. Let $\xi_X$ be a $G$-labelled covering. The map $\mathcal{M}_X:\pi_1(X,x_0)\to G$ gives rise to a unique homotopy class of maps $cl_X:(X,x_0)\to(BG,b_0)$ so that $\mathcal{M}_X=cl_{X*}:\pi_1(X,x_0)\to \pi_1(BG,b_0)=G$. Let $w(\xi_X)\in H^k(X,A)$ be the pullback of the class $w$ through $cl_X$. It is easy to check that the class $w$ thus constructed is characteristic.
\end{example}

\begin{example}
\label{example:abelian}
Let $G$ be a finitely generated abelian group. Let $n$ be any natural number and suppose that $G/nG$ is isomorphic to $(\mathbf{Z}_n)^k$ for some $k$ (this will be automatically true if $n$ is prime for example). Fix such an isomorphism of $G/nG$ with $(\mathbf{Z}_n)^k$. For a $G$-labelled covering $\xi_X$ let $c\in H^1(X,G)$ be the cohomology class obtained from $\mathcal{M}_X$ by identification $Hom(\pi_1(X,x_0),G)\cong Hom(H_1(X),G)\cong H^1(X,G)$ (the first equality follows because $G$ is abelian). One can also think of this class as the Cech cohomology class that defines the principal $G$-bundle associated with the covering $\xi_X$.

Let now $p_j:G\to \mathbf{Z}_n$ be the composite of the quotient map $G\to G/nG$ and the projection from $(\mathbf{Z}_n)^k$ to the $j$-th factor in the product. Let $w(\xi_X)$ be the cup product of the images of $c$ under the maps $p_{j*}:H^1(X,G)\to H^1(X,\mathbf{Z}_{n})$, i.e. $w(\xi_X)=p_{1*}(c)\cup \ldots \cup p_{k*}(c) \in H^k(X,\mathbf{Z}_{n})$. Once again it is easy to see that the class $w$ is characteristic.
\end{example}

\begin{example}
\label{example:stiefel}
Every $n$-sheeted covering $\xi_X$ gives rise to an $n$-dimensional real vector bundle by the change of fiber over point $x\in X$ from $p_X^{-1}(x)$ to the real vector space spanned by the points of $p_X^{-1}(x)$. Stiefel-Whitney classes of this bundle give rise to characteristic classes for the category of $n$-sheeted coverings. 
\end{example}

Example \ref{example:classifying} is a very general way of constructing characteristic classes for $G$-labelled coverings (indeed, all characteristic classes can be produced this way). It involves however computing the cohomology of a group. Example \ref{example:abelian} is an extremely simple construction, but it will prove powerful enough for our purposes: finding a topological obstruction to inducing a given covering with abelian monodromy group from a covering over a space of a small dimension. Example \ref{example:stiefel} has been used in \cite{Arn1}. A variation on example \ref{example:classifying} with group $S(n)$ and with coefficients taken in an $S(n)$-module $\mathbf{Z}$ with action given by the sign representation $S(n)\to Aut(\mathbf{Z})\cong \mathbf{Z}_{2}$ has been used in \cite{Vas1} (note however that our definition of characteristic classes is too restrictive to include this as an example).

\begin{remark}Characteristic classes defined for all $n$-sheeted coverings are rather weak in distinguishing coverings, whose monodromy group consists only of even permutations. For example consider the degree 3 covering $\xi_X$ over the circle $X=S^1$ given by $p:S^1\rightarrow S^1$, $p(z)=z^3$ (we think of the circle as the circle of unit-length complex numbers). Then every characteristic class $w$ (with any coefficients) vanishes on $\xi_X$. Indeed, consider figure eight $Y=S^1\vee S^1$ with the base point $y_0$ being the common point of the two circles. Let $a,b$ denote the two loops corresponding to the two circles in figure eight. Now consider the covering $\xi_Y$ with monodromy representation sending $[a]\in \pi_1(Y,y_0)$ to the permutation $(1 2)\in S(3)$ and $[b]\in \pi_1(Y,y_0)$ to $(2 3)$. Then our covering $\xi_X$ is induced from $\xi_Y$ by mapping $g:X\rightarrow Y$ sending the loop that goes around the circle $X$ to the path $a b a^{-1} b^{-1}$ in $Y$ (indeed, $(1 2 3)=(1 2) (2 3) (1 2)^{-1} (2 3)^{-1}$), and hence $w(\xi_X)=g^*(w(\xi_Y))$. However $g^*:H^1(Y)\rightarrow H^1(X)$ is clearly the zero map, so no characteristic class for 3-sheeted coverings can distinguish $\xi_X$ from the trivial covering.

By taking Cartesian product of $m$ copies of this example, we get a covering over $m$-dimensional torus, which can be induced from a covering over some space (a product of $m$ figure-eights) through a map that induces the zero map on the reduced cohomology ring. Thus all characteristic classes defined for $3^m$-sheeted coverings must vanish on it. Later we will show that this covering can't be induced from any covering over a space of dimension smaller than $m$, thus showing it is very far from being trivial.
\end{remark}

\section{Inducing coverings from spaces of low dimension}

\subsection{From coverings to coverings with the same monodromy group}

\begin{lemma}
\label{lemma:dim}
Let $h:(X,x_0)\rightarrow (Y,y_0)$ be a mapping between pointed spaces. Then there exists a pointed space $(\tilde{Y},\tilde{y}_0)$ and maps $g:(\tilde{Y},\tilde{y}_0)\rightarrow (Y,y_0)$, $f:(X,x_0)\rightarrow(\tilde{Y},\tilde{y}_0)$ so that $h=g\circ f$, the mapping $g$ is a covering map and the homomorphism $f_*:\pi_1(X,x_0)\rightarrow \pi_1(\tilde{Y},\tilde{y_0})$ induced by $f$ on the fundamental groups is surjective.
\end{lemma}

The proof of the lemma is an explicit construction: we define the covering $g:(\tilde{Y},\tilde{y}_0)\rightarrow (Y,y_0)$ as the covering over $(Y,y_0)$ that corresponds under the Galois correspondence for coverings to the subgroup $h_*(\pi_1(X,x_0))$ in $\pi_1(Y,y_0)$ (that is we ``unwind'' all the loops in $Y$ that are not the images of loops from $X$). Then we define $f:(X,x_0)\rightarrow(\tilde{Y},\tilde{y_0})$ as the unique lifting of the mapping $h:(X,x_0)\rightarrow(Y,y_0)$ through the map $g$ (which exists because $h_*(\pi_1(X,x_0))=g_*(\pi_1(\tilde{Y},\tilde{y_0}))$). It remains to check that the mapping $f$ thus defined induces a surjective homomorphism on the fundamental groups, i.e. that $f_*(\pi_1(X,x_0))=\pi_1(\tilde{Y},\tilde{y}_0)$. Because $g_*$ is injective, this is equivalent to verifying that $g_*(f_*(\pi_1(X,x_0)))=g_*(\pi_1(\tilde{Y},\tilde{y}_0))$. This is true, since both sides are equal to $h_*(\pi_1(X,x_0))$: the left side, because $g_*\circ f_*=h_*$ and the right side --- by construction.

\begin{lemma}
\label{lemma:labelling}
Let $f:(X,x_0)\to (Y,y_0)$ be a map inducing a surjective homomorphism of fundamental groups. Let covering $\xi_X$ over $X$ be induced from a covering $\xi_Y$ over $Y$ by means of the map $f$. Suppose that the covering $\xi_X$ can be $G$-labelled. Then the covering $\xi_Y$ can also be $G$-labelled in a way that the $G$-labelled covering $\xi_X$ is induced from the $G$-labelled covering $\xi_Y$ (as a $G$-labelled covering).
\end{lemma}
\begin{proof}
Let $M_X$ and $M_Y$ be the monodromy representations of the coverings $\xi_X$ and $\xi_Y$ and let $L_X$ be the labelling $L_X:M_X(\pi_1(X,x_0))\to G$. We define the labelling $L_Y:M_Y(\pi_1(Y,y_0))\to G$ as follows: a permutation in $M_Y(\pi_1(Y,y_0))$ is realized as the monodromy along some element $\alpha\in \pi_1(Y,y_0)$. Since $f_*:\pi_1(X,x_0)\to (Y,y_0)$ is surjective by assumption, $\alpha = f_*(\beta)$ for some $\beta\in \pi_1(X,x_0)$. We define the image of the permutation we started with under $L_Y$ to be $L_X(M_X(\beta))$. This definition doesn't depend on the choice of $\alpha$ or its preimage $\beta$ since the covering $\xi_X$ is induced from the covering $\xi_Y$ by means of $f$ (in fact the monodromy along $\beta$ doesn't depend on the choice of $\beta$: it is the same as the permutation we started with after identifying the fiber of $p_Y$ over $y_0$ with the fiber of $p_X$ over $x_0$).
\end{proof}

We will be mainly interested in the following corollary from these lemmas:
\begin{corollary}
\label{corollary:induce}
Let $\xi_X$ be a $G$-labelled covering. Suppose it can be induced (as a covering, not necessarily as a $G$-labelled covering) from a covering over a space $Y$ of dimension $\leq m$. Then it can also be induced from a $G$-labelled covering over a space $\tilde{Y}$ of dimension $\leq m$ by means of a map $f:(X,x_0)\rightarrow (\tilde{Y},\tilde{y}_0)$ with the property that $f_*:\pi_1(X,x_0)\rightarrow \pi_1(\tilde{Y},\tilde{y}_0)$ is surjective. 
\end{corollary}

\begin{proof}
Let the covering $\xi_X$ be induced from the covering $\xi_Y$ by means of the map $h:(X,x_0)\rightarrow(Y,y_0)$. From lemma \ref{lemma:dim} above one can construct pointed space $(\tilde{Y},\tilde{y_0})$ and maps $g:(\tilde{Y},\tilde{y_0})\rightarrow (Y,y_0)$ and $f:(X,x_0)\rightarrow (\tilde{Y},\tilde{y_0})$ so that $g$ is a covering map, $f$ induces surjective homomorphism on the fundamental groups and $h=g\circ f$. The space $\tilde{Y}$ is of dimension $\leq m$, because it covers the space $Y$. The covering $\xi_X$ is induced from the covering $g^*\xi_Y$ on $\tilde{Y}$ by means of $f$. According to lemma \ref{lemma:labelling} the covering $g^*\xi_Y$ can be $G$-labelled so that the $G$-labelled covering $\xi_X$ is induced from it by means of $f$ as a $G$-labelled covering.
\end{proof}

\subsection{Equivalent coverings}

It turns out that some essential properties of a covering depend only on the abstract isomorphism class of its monodromy representation, rather than on the isomorphism class of the covering itself. The purpose of this section is to make one case of this observation precise.

\begin{definition}
Let $\xi_{X,1}$ and $\xi_{X,2}$ be two coverings over the same space $X$ with monodromy representations $M_{X,1}:\pi_1(X,x_0)\rightarrow S(p_{X,1}^{-1}(x_0))$ and $M_{X,2}:\pi_1(X,x_0)\rightarrow S(p_{X,2}^{-1}(x_0))$ respectively.  The coverings $\xi_{X,1}$ and $\xi_{X,2}$ are called \textbf{equivalent} if there exists an isomorphism $$g:M_{X,1}(\pi_1(X,x_0))\rightarrow M_{X,2}(\pi_1(X,x_0))$$ making the following diagram commutative: 

$$\xymatrixcolsep{0pc}\xymatrix{ & \pi_1(X,x_0) \ar[dl]_{M_{X,1}} \ar[dr]^{M_{X,2}} &  \\
      M_{X,1}(\pi_1(X,x_0))   \ar[rr]_g     &    & M_{X,2}(\pi_1(X,x_0))
}
$$

\end{definition}

One can think of equivalent coverings as coverings that can be obtained from each other by means of change of the fiber.

This definition is important for us because of the following lemma.

\begin{lemma}
\label{lemma:equiv}
Suppose that the covering $\xi_{X,1}$ is induced form the covering $\xi_{Y,1}$ by means of map $f:(X,x_0)\rightarrow(Y,y_0)$ that induces surjective homomorphism on fundamental groups. Let $\xi_{X,2}$ be a covering on $X$ which is equivalent to the covering $\xi_{X,1}$. Then there exists a covering $\xi_{Y,2}$ over $Y$ so that $\xi_{X,2}=f^*(\xi_{Y,2})$, that is the covering $\xi_{X,2}$ is also induced from a covering over the same space $Y$. 
\end{lemma}

\begin{proof}
Since the covering $\xi_{X,1}$ is induced from the covering $\xi_{Y,1}$, we can identify the corresponding fibers $p_{X,1}^{-1}(x_0)$ and $p_{Y,1}^{-1}(y_0)$. Let $I:S(p_{X,1}^{-1}(y_0))\rightarrow S(p_{Y,1}^{-1}(x_0))$ be the corresponding identification of the permutation groups. Then $I\circ M_{X,1}=M_{Y,1}\circ f_* : \pi_1(X,x_0)\rightarrow S(p_{Y,1}^{-1}(y_0))$ (because $\xi_{X,1}=f^*(\xi_{Y,1})$). Let $g:M_{X,1}(\pi_1(X,x_0))\rightarrow M_{X,2}(\pi_1(X,x_0))$ denote the isomorphism showing that the coverings $\xi_{X,1}$ and $\xi_{X,2}$ are equivalent. We define action $M_{Y,2}:\pi_1(Y,y_0)\rightarrow S(p_{X,2}^{-1}(x_0))$ as follows: let $\alpha\in \pi_1(Y,y_0)$ be any element. Since $f_*:\pi_1(X,x_0)\rightarrow \pi_1(Y,y_0)$ is surjective, we can choose $\beta\in \pi_1(X,x_0)$ so that $f_*\beta=\alpha$. Define $M_{Y,2}(\alpha)$ as $M_{X,2}(\beta)$. This definition is independent of the choice of the preimage $\beta$ of $\alpha$, because $M_{X,2}(\beta)=g(M_{X,1}(\beta))=g(I^{-1}(M_{Y,1}(f_*\beta)))=g(I^{-1}(M_{Y,1}(\alpha)))$, and the right hand side is independent of the choice.

By construction of section \ref{section:monodromy}, the action $M_{Y,2}$ defines a covering $\xi_{Y,2}$ for which $M_{Y,2}$ is the monodromy action, and since $M_{Y,2}\circ f_* = M_{X,2}$ by definition, the covering $\xi_{X,2}$ is induced from it by means of the map $f$.
\end{proof}

\begin{remark}
This lemma shows in particular that if $\xi_X$ is a covering with connected total space $\tilde{X}$ then it is equivalent to its associated Galois covering (i.e. the minimal Galois covering that dominates $\xi_X$).
\end{remark}

\subsection{Domination}

We will later need the notion of one covering being more ``complicated'' than another covering:

\begin{definition}
We say that a covering $\xi_{X}^1$ ($p_1:\tilde{X}^1\to X$) \textbf{dominates} the covering $\xi_{X}^2$ ($p_2:\tilde{X}^2\to X$) if there exists a covering map $p:\tilde{X}^1\to \tilde{X}^2$ making the diagram below commutative

$$\xymatrix{ \tilde{X}^1 \ar[dr]_{p_1} \ar[r]^p&  \tilde{X}^2 \ar[d]^{p_2} \\
              &  X 
}
$$
\end{definition}

\begin{lemma}
\label{lemma:domination}
Suppose that the covering $\xi_X$ on $X$ can be induced from a covering $\xi_Y$ on $Y$ by means of a map $f:X\to Y$ inducing a surjective homomorphism on fundamental groups. Suppose also that the covering $\xi_X$ dominates a covering $W\to X$:

$$\xymatrix{ \tilde{X} \ar[dr]_{p_X} \ar[r] & W \ar[d]   \\ 
              & X  
}
$$
(the maps $\tilde{X}\to W$ and $W\to X$ in the diagram above are covering maps).

Then the covering $W\to X$ can be induced by means of the map $f$ from a covering on $Y$.
\end{lemma}

\begin{proof}

The proof consists of an explicit construction of the covering on $Y$ from which the covering $W\to X$ is induced by means of $f$ and is similar to the proof of lemma \ref{lemma:equiv}.

Let $x_0$ be the base point in $X$ and let $y_0=f(x_0)$ be the base point in $Y$. Denote by $F$ the fiber $p_Y^{-1}(y_0)$ of $\xi_Y$ over $y_0$. Since $\xi_X$ is induced from $\xi_Y$, the fiber of $\xi_X$ over $x_0$ can be naturally identified with $F$ as well. Let $M_F^X$ denote the monodromy action of $\pi_1(X,x_0)$ on $F$ corresponding to the covering $\xi_X$ and let $M_F^Y$ denote the monodromy action of $\pi_1(Y,y_0)$ on $F$ corresponding to the covering $\xi_Y$. Since $\xi_X=f^*(\xi_Y)$, we have $M_F^X=M_F^Y\circ f_*$.

Denote by $Q$ the fiber of $W\to X$ over $x_0$. Let $M_Q^X$ denote the monodromy action of $\pi_1(X,x_0)$ on $Q$. Let $q:F\to Q$ denote the restriction of the covering map $\tilde{X}\to W$ to the fiber $F$. For every $\beta\in\pi_1(X,x_0)$ the diagram

$$\xymatrixcolsep{6pc}\xymatrix{ F \ar[d]^q \ar[r]^{M_F^X(\beta)} &  F\ar[d]^{q} \\
             Q \ar[r]^{M_Q^X(\beta)} &  Q
}
$$
commutes (this is equivalent to the fact that $\xi_X$ dominates $W\to X$).

 We will now introduce an action $M_Q^Y$ of $\pi_1(Y,y_0)$ on $Q$ satisfying $M_Q^X=M_Q^Y\circ f_*$. This action will give rise to the required covering on $Y$ from which $W\to X$ is induced.

Let $\alpha\in \pi_1(Y,y_0)$ be any element. Let $\beta$ be any of its preimages under $f_*$. We define $M_Q^Y(\alpha)$ to be $M_Q^X(\beta)$. This element in fact does not depend on the choice of $\beta$. Indeed, let $\beta'$ be another preimage of $\alpha$. Then $M_F^X(\beta)$ and $M_F^X(\beta')$ are equal, since both are equal to $M_F^Y(\alpha)$. But then $M_Q^X(\beta)$ and $M_Q^X(\beta')$ must be the same, since both make the diagram

$$\xymatrixcolsep{6pc}\xymatrix{ F \ar[d]^q \ar[r]^{M_F^X(\beta)=M_F^X(\beta')} &  F\ar[d]^{q} \\
             Q \ar[r]^{M_Q^X(\beta), M_Q^X(\beta')} &  Q
}
$$

commutative and $q:F\to Q$ is surjective.

The facts that $M_Q^Y$ thus defined is an action and that $M_Q^X=M_Q^Y\circ f_*$ are easy to verify.
\end{proof}

For us the following corollary will be important:

\begin{corollary} \label{corollary:dominated}
Suppose the covering $\xi_{X}$ over space $X$ can be induced from a covering over a space of dimension $\leq k$. Then every covering it dominates can also be induced from a space of dimension $\leq k$.
\end{corollary}
\begin{proof}
Lemma \ref{lemma:dim} implies that if a covering on $X$ can be induced from a covering over a space of dimension $\leq k$, then is can be induced from a space of dimension $\leq k$ by means of a map that induces surjective homomorphism on fundamental groups. Lemma \ref{lemma:domination} above implies then that any covering that it dominates can also be induced from a space of dimension $\leq k$. 
\end{proof}

\subsection{ Example}

Before stating general results, we will go back to example in the remark in section \ref{sectionchar}: the covering $\xi_X$ over the space $X=S^1$ given by the map $p_X:S^1\rightarrow S^1$ sending $z\in S^1$ to $p_X(z)=z^3$ (we think of $S^1$ as of unit complex numbers). This covering can be $\mathbf{Z}_3$-labelled in an obvious way (in fact in two ways - we have to choose one of them). Consider the covering $\xi_X\times \ldots \times \xi_X$ over $X^m$, the $m$-dimensional torus. It can be $\mathbf{Z}_3\times \ldots \times \mathbf{Z}_3$ - labelled in an obvious way. The characteristic class from example \ref{example:abelian} with coefficients in $\mathbf{Z}_3$ having degree $m$ doesn't vanish for this covering.

According to corollary \ref{corollary:induce}, if this covering could be induced from a covering on a space of dimension $<m$, it would be also possible to induce the corresponding $\mathbf{Z}_3^m$-labelled covering from a $\mathbf{Z}_3^m$-labelled covering on a space of dimension $<m$. But then naturality of characteristic classes would imply that any degree $m$ characteristic class for the category of  $\mathbf{Z}_3^m$-labelled coverings must vanish on it. Thus the covering we are considering can't be induced from a covering of dimension $<m$.

\section{Coverings over tori}
\label{section:tori}

We now proceed to proving a general result answering the questions: what coverings over a torus $\mathbf{T}=(S^1)^n$ can be induced from coverings over spaces of dimension $k$.
 The result is as follows:

\begin{theorem} 
\label{theorem:mainresult}
  A covering $\xi_{\mathbf{T}}$ over a torus $\mathbf{T}$ can be induced from a covering over $k$-dimensional space if and only if the monodromy group of the covering $\xi_{\mathbf{T}}$ (considered as an abstract group) can be represented as a direct sum of $k$ cyclic groups. In the case $\xi_\mathbf{T}$ can be induced from a covering over some space of dimension $k$, it can also be induced from a covering over $k$-dimensional torus $\left(S^1\right)^k$.
\end{theorem}

 Before we prove this result, we will describe a normal form for the equivalence class of a covering over a torus.

Let $\xi_n$ denote the covering over the circle $S^1$ given by the map $p_m:S^1\rightarrow S^1$ sending $z\in S^1$ to $z^m\in S^1$ (we think of $S^1$ as of the circle of unit length complex numbers). Let also $\xi_\infty$ denote the covering given by the map $p_\infty: \mathbf{R} \rightarrow S^1$ sending $x\in \mathbf{R}$ to $e^{i x}\in S^1$. Then the following lemma holds:

\begin{lemma}\label{lemma:normal_form}
Every covering over a torus $\mathbf{T}=(S^1)^n$ is equivalent to the covering $\xi_1^s\times\xi_{m_1}\times\xi_{m_2}\times\ldots\times\xi_{m_t}\times\xi_\infty^r$ for some integer numbers $s,t,r\geq 0$ with $s+t+r=n$ and natural numbers $m_1,\ldots,m_t\geq 2$ satisfying the divisibility condition $m_1|m_2|\ldots|m_t$.
\end{lemma}

\begin{remark}
Note that the covering $\xi_1^s$ in the representation above is just the trivial degree 1 covering over the $s$-dimensional torus.
\end{remark}

\begin{proof}
Let $t_0\in \mathbf{T}$ be an arbitrary point of the torus and let $M_\mathbf{T}:\pi_1(\mathbf{T},t_0)\rightarrow S(p_\mathbf{T}^{-1}(t_0))$ be the monodromy representation of the covering $\xi_\mathbf{T}$. Choose a basis $e_1,\ldots,e_n$ for the free abelian group $\pi_1(\mathbf{T},t_0)$ so that $\pi_1(\mathbf{T},t_0)$ gets identified with the group $\mathbf{Z}^n$ spanned on the generators $e_1,\ldots,e_n$. The kernel of the homomorphism $M_\mathbf{T}$ is a subgroup of $\mathbf{Z}^n$, hence it is also a free abelian group. Choose a basis $E_1,\ldots,E_q$ for the kernel. We can express each vector $E_i$ as an integer linear combination of the basis vectors $e_j$: $E_i=\sum_j{a_{i,j} e_j}$. By a suitable change of bases $e$ and $E$ for the lattice and its sublattice we can bring the matrix $(a_{i,j})$ to its Smith normal form, that is after a change of bases we can get that $E_1=e_1, \ldots, E_s=e_s, E_{s+1}=m_1\cdot e_{s+1}, E_{s+2}=m_2\cdot e_{s+2},\ldots,E_q=m_t\cdot e_q$, where $s$ is the number of ones in the Smith normal form of the matrix, $q=s+t$ and $m_1,\ldots,m_t \geq 2$ are integers with the divisibility property $m_1|m_2|\ldots|m_t$.

This means that the monodromy representation $M_\mathbf{T}$, considered as a mapping onto its image, is isomorphic to the product of trivial maps $\mathbf{Z}\rightarrow 0$ in the first $s$ coordinates, quotient maps $\mathbf{Z}\rightarrow \mathbf{Z}_{m_i}$ in the next $t$ coordinates and identity maps $\mathbf{Z}\rightarrow \mathbf{Z}$ in the remaining $r=n-s-t$ coordinates. Thus the covering $\xi_\mathbf{T}$ is equivalent to the covering $\xi_1^s\times\xi_{m_1}\times\xi_{m_2}\times\ldots\times\xi_{m_t}\times\xi_\infty^r$.
\end{proof}

We now proceed to the proof of theorem \ref{theorem:mainresult}
\begin{proof} 

Let $G$ denote the monodromy group of the covering $\xi_\mathbf{T}$. From lemma \ref{lemma:normal_form} above the covering $\xi_{\mathbf{T}}$ is equivalent to the covering $\xi_1^s\times\xi_{m_1}\times\xi_{m_2}\times\ldots\times\xi_{m_t}\times\xi_\infty^r$ for some integers $s,t,r\geq 0$ with $s+t+r=n$ and natural numbers $m_1,\ldots,m_t$ satisfying $m_1|m_2|\ldots|m_t$.

The monodromy group $G$ of this covering is isomorphic to the sum of $k=t+r$ cyclic groups: $G=\mathbf{Z}_{m_1}\oplus\ldots\oplus\mathbf{Z}_{m_t}\oplus\mathbf{Z}^r$. This group cannot be represented as a sum of less than $k=t+r$ cyclic groups ($k$ being the dimension of the $\mathbf{Z}_p$-vector space $G/p G$ for $p$ being some prime divisor of $p_1$).

The covering $\xi_1^s\times\xi_{m_1}\times\xi_{m_2}\times\ldots\times\xi_{m_t}\times\xi_\infty^r$ clearly can be induced from the covering $\xi_{m_1}\times\xi_{m_2}\times\ldots\times\xi_{m_t}\times\xi_\infty^r$ over $k=t+r$-dimensional torus via the projection on the last $k$ coordinates. This projection map induces a surjective homomorphism on the fundamental groups, hence the covering $\xi_\mathbf{T}$, being equivalent to the covering $\xi_1^s\times\xi_{m_1}\times\xi_{m_2}\times\ldots\times\xi_{m_t}\times\xi_\infty^r$, also can be induced from a covering over $k$-dimensional torus according to lemma \ref{lemma:equiv}. 

It remains to show that the covering $\xi_\mathbf{T}$ can't be induced from a covering over a space of dimension $<k$. Suppose to the contrary that it can. Lemma \ref{lemma:equiv} then tells us that the equivalent covering $\xi_1^s\times\xi_{n_1}\times\xi_{n_2}\times\ldots\times\xi_{n_t}\times\xi_\infty^r$ also can be induced from a space of dimension $<k$.

Covering  $\xi_1^s\times\xi_{n_1}\times\xi_{n_2}\times\ldots\times\xi_{n_t}\times\xi_\infty^r$ can be $G$-labelled in a natural way ($G$ being its monodromy group $\mathbf{Z}_{m_1}\oplus\ldots\oplus\mathbf{Z}_{m_t}\mathbf{Z}\oplus\mathbf{Z}^r$). Corollary \ref{corollary:induce} then tells us that this $G$-labelled covering can be induced from a $G$-labelled covering over a space of dimension $<k$. In particular the value of any degree $k$ characteristic class for the category of $G$-labelled coverings must vanish on it.

However if we consider the characteristic class from example \ref{example:abelian} with coefficients in $\mathbf{Z}_{m_1}$, we find that it doesn't vanish!

\end{proof}

\begin{definition}
The minimal number $k$ such that a finitely-generated abelian group $G$ can be represented as a direct sum of $k$ cyclic groups is called the \textbf{rank} of $G$.
\end{definition}

Corollary \ref{corollary:dominated} implies that a slightly stronger result holds as well:

\begin{theorem}
Suppose the covering $\xi_{\mathbf{T}}$ over a torus $\mathbf{T}$ has monodromy group of rank $k$. Then it is not dominated by any covering that can be induced from a space of dimension  strictly smaller than $k$.
\end{theorem}

Later, in algebraic context, we will need a version of this result dealing with a tower of coverings dominating a given one. We will formulate the result now:

\begin{theorem}
\label{theorem:mainresult_comp}
Suppose $\xi_T$ is a covering over the torus $T$ with monodromy group of rank $k$. Let $f:T_s\to T$ be a covering map over $T$ that factors as the composition of covering maps $T_s\xrightarrow{f_s} T_{s-1} \to \ldots \to T_1 \xrightarrow{f_1} T_0=T$ and assume that each covering $f_i:T_i\to T_{i-1}$ can be induced from a covering over a space of dimension $\leq k_i$. Then the rank of monodromy group of the covering $f^*\xi_T$ is at least $k-\sum k_i$. In particular if $\sum k_i<k$, the covering $f^*\xi_T$ is not trivial.
\end{theorem}

A couple of lemmas will be needed to prove this theorem:

\begin{lemma}
Let $A,B,C$ be three finitely generated abelian groups that fit into the exact sequence $$0\to A \to B \to C \to 0$$ Then $\rk B\leq \rk A+\rk B$.
\end{lemma}
\begin{proof}
Let $p$ be a prime number such that $\rk B=\dim B\otimes \mathbf Z_p$. The exact sequence of $\mathbf Z_p$ vector spaces $$A\otimes \mathbf Z_p \to B\otimes \mathbf Z_p \to C \otimes \mathbf Z_p \to 0$$ shows that $\dim B\otimes \mathbf Z_p \leq \dim A\otimes \mathbf Z_p + \dim C\otimes \mathbf Z_p$. Since $\rk A\geq \dim A\otimes \mathbf Z_p$ and similarly for $C$, we have $\rk B = \dim B\otimes \mathbf Z_p\leq \rk A+\rk C$
\end{proof}

This algebraic lemma is applicable in topological context due to the following:

\begin{lemma}
Let $(X_3,x_3)\xrightarrow{f} (X_2,x_2)\xrightarrow{g} (X_1,x_1)$ be two covering maps and assume that $X_2$ is connected and the monodromy group of $g\circ f$ is abelian. Let $G(f),G(g),G(g\circ f)$ be the monodromy groups of the coverings $f,g,g\circ f$ respectively. These monodromy groups fit into an exact sequence $$0 \to G(f) \to G(g\circ f) \to G(g) \to 0 $$
\end{lemma}

\begin{proof}
Let $M_{g}$ and $M_{g\circ f}$ denote the monodromy representations of $\pi_1(X_1,x_1)$ on the permutation groups $S(g^{-1}(x_1))$ and $S((g\circ f)^{-1}(x_1))$ and let $M_f$ denote the monodromy representation of $\pi_1(X_2,x_2)$ on $S(f^{-1}(x_2))$. 

The map $f$ maps the fiber $(g\circ f)^{-1}(x_1)$  to $g^{-1}(x_1)$ and hence induces a map $f_*: S((g\circ f)^{-1}(x_1)) \to S(g^{-1}(x_1))$. The restriction of this map to $G(g\circ f)$ maps $G(g\circ f)$ onto $G(g)$ because the following diagram commutes:

$$\xymatrixcolsep{4pc}\xymatrix{
\pi_1(X_1,x_1) \ar@{->>}[r]^-{M_{g\circ f}} \ar@{->>}[dr]+UL+<2ex,0ex>_-{M_g}& G(g\circ f) \subset  S((g\circ f)^{-1}(x_1)) \ar@<4ex>[d]^-{f_*} \\
                                  & G(g) \subset  S(g^{-1}(x_1))
}$$

The kernel of the restriction of $f_*$ to $G(g\circ f)$ is equal to  $M_{g\circ f}(\ker M_g)$. Now we claim that $M_{g\circ f}(\ker M_g)$ is isomorphic to $G(f)$. 

Let $r:M_{g\circ f}(\ker M_g)\to G(f)$ be the following map: it sends a permutation of the fiber $(g\circ f)^{-1}(x_1)$ that belongs to $M_{g\circ f}(\ker M_g)$ to its restriction to the fiber $f^{-1}(x_2)$. This restriction is a permutation of $f^{-1}(x_2)$ that lies in $G(f)$ because if the initial permutation is realized as the monodromy along a loop $\gamma$ whose class in $\pi_1(X_1,x_1)$ is in the kernel of $M_g$, then this loop lifts to a loop based at $x_2$ and the monodromy of $f$ realized along this loop is the required permutation in $G(f)$.

The map $r$ is clearly a group homomorphism. It is onto because a permutation in $G(f)$ can be realized as the monodromy of $f$ along a loop in $X_2$ based at $x_2$. The monodromy of $g\circ f $ along the image of this loop under $g$ is a preimage of the permutation we started with under $r$. 

Finally we want to show that the map $r$ is one-to-one. Suppose that a permutation in $M_{g\circ f}(\ker M_g)$ restricts to a trivial permutation on the fiber $f^{-1}(x_2)$. Since this permutation is in $M_{g\circ f}(\ker M_g)$, it can be realized as the monodromy of $g\circ f $ along a loop $\gamma$ in $X_1$ based at $x_1$ that lifts to a closed loop in $X_2$ with any choice of the lift of $x_1$ to a point in $g^{-1}(x_1)$. Let $\alpha$ be one such lift with $\alpha(0)=\tilde{x}_2\in g^{-1}(x_1)$. It is enough to show that the monodromy of $f$ along this loop is trivial. Choose a path $\beta$ in $X_2$ connecting $x_2$ to $\tilde{x}_2$. The monodromy of $g\circ f$ along the loop $g_*(\beta \alpha \beta^{-1})$ is  $M_{g\circ f}(g_*\beta g_*\alpha g_*\beta ^{-1})=M_{g\circ f}(g_*\beta) M_{g\circ f}(g_*\alpha) M_{g\circ f}(g_*\beta) ^{-1}=M_{g\circ f}(\gamma)$ since the monodromy group of $g\circ f$ is abelian. In particular the monodromy of $g\circ f$ along $g_*(\beta \alpha \beta^{-1})$ resricts to the trivial permutation on $f^{-1}(x_1)$, which means that the monodromy of $f$ along $\alpha$ restricts to the trivial permutation of $f^{-1}(\tilde{x}_2)$. Since this conclusion holds for any lift of $\gamma$ to a loop $\alpha$ based at any point $\tilde{x}_2\in g^{-1}(x_1)$, the monodromy of $g\circ f $ along $\gamma$ is trivial.
\end{proof}

This lemma can be applied to prove the following claim about coverings over a torus:

\begin{lemma}
Let $T_k\xrightarrow{f_k} T_{k-1} \to \ldots \to T_1 \xrightarrow{f_1} T_0$ be a sequence of covering maps, where $T_0$ is a torus, $T_1,\ldots,T_{k-1}$ are connected, while $T_k$ is not necessarily connected. Then rank of the monodromy group of the composite covering $f_k\circ\ldots\circ f_1$ is smaller than or equal to the sum of the ranks of the monodromy groups of the coverings $f_i$.
\end{lemma}

\begin{proof}
The proof is a simple induction on $k$ and based on the fact that the fundamental group of a torus is a finitely generated abelian group and the previous two lemmas. 
\end{proof}

Finally this allows us to prove theorem \ref{theorem:mainresult_comp}:
\begin{proof}
Theorem \ref{theorem:mainresult} implies that the rank of monodromy group of the covering $f_i:T_i\to T_{i-1}$ is at most $k_i$. If we denote by $\tilde{k}$ the rank of the monodromy of the covering $f^*\xi_T$, then the lemma above implies that the rank of the composition of the covering $f^*\xi_T$ with $f$ is at most $\tilde k+ \sum k_i$. On the other hand this rank is at least $k$, since this covering dominates the covering $\xi_T$. Hence $\tilde k\geq k-\sum k_i$.
\end{proof}

\section{Klein's Resolvent Problem}
\label{section:formulations}

\subsection{Formulation}

Klein's resolvent problem is the problem of deciding whether a given algebraic equation depending on several independent parameters admits a rational transformation transforming it into an equation depending on a smaller number of algebraically independent parameters (see \cite{Cheb1}).

More precisely we can introduce the following definition:

\begin{definition} \label{definition:rationally_induced}
An algebraic function $\mathbf{z}$ defined over a Zariski open subset of a variety $X$ is said to be \textbf{rationally induced} from an algebraic function $\mathbf{w}$ defined over a Zariski open subset of a variety $Y$ if there exists a Zariski open subset $U$ of $X$, a rational morphism $r$ from $X$ to $Y$ and a rational function $R$ on $X\times \mathbf{C}$ such that:
\begin{itemize}
\item{the function $\mathbf{z}(x)$ is defined for all $x\in U$}
\item{the function $R(x,\mathbf{w}(r(x)))$ is defined for all $x\in U$}
\item{the function $\mathbf{z}(x)$ is a branch of $R(x,\mathbf{w}(r(x)))$ for $x\in U$}
\end{itemize}
(it is assumed that the functions $r(x)$, $\mathbf{w}(r(x))$ and $R(x,\mathbf{w}(r(x)))$ are all defined for $x\in U$)
\end{definition}

This definition can be used to formulate Klein's resolvent problem precisely:

\begin{question} \label{question:algebraic_functions}
Given an algebraic function $\mathbf{z}$ on an irreducible variety $X$ what is the smallest number $k$ such that the function $\mathbf{z}$ can be rationally induced from an algebraic function $\mathbf{w}$ on some variety $Y$ of dimension $\leq k$?
\end{question} 

As most of the arguments for treating this question will be geometrical in nature, we would like to formulate this question in geometric terms. Instead of an algebraic function we will talk of a branched covering defined by it (see section \ref{section:algebraic_functions} below). Question \ref{question:algebraic_functions} can then be reformulated:

\begin{definition}
A branched covering $\xi_X$ over an irreducible variety $X$ is \textbf{rationally induced} from a branched covering $\xi_Y$ over a variety $Y$ if there exists a dominant rational morphism $f:X\to Y$ and a Zariski open subset $U$ of $X$ such that the restriction of the branched covering $\xi_X$ to $U$ is a covering and this covering is dominated by the restriction of the branched covering $f^*(\xi_Y)$ to $U$ (the mapping $f$ is assumed to be defined everywhere on $U$)
\end{definition}


\begin{question} \label{question:branched_coverings}
Let $\xi_X$ be a branched covering over an irreducible variety $X$. For what numbers $k$ there exists a branched covering $\xi_Y$ over an irreducible variety $Y$ of dimension $k$, such that the branched covering $\xi_X$ can be rationally induced from it?
\end{question}

The questions above can be formulated algebraically using the language of field extensions. This was done for instance in \cite{Buh1},\cite{Buh2}:

\begin{question} \label{question:fields} Let $E /K$ be a finite degree extension of fields and suppose $K$ is of finite transcendence degree over $\mathbf{C}$. For what numbers $k$ there exists a field extension $e/\mathbf{C}$ of transcendence degree $k$ so that $E\subset K(e)$?
\end{question}

In other words we are trying to get the extension $E/K$ by adjoining to the field of rationality $K$ ``irrationalities" (elements of $e$) depending on as few parameters as possible (the number of parameters being the transcendence degree of $e$ over $\mathbf{C}$). 

The minimal number of such parameters is called the \textbf{essential dimension} of the extension $E/K$.

Hilbert has formulated a version of Klein's resolvent problem as problem 13 in his famous list. While he hasn't specified an exact formulation of this problem one possible way to formulate his question is the following:

\begin{question} \label{question:hilbert}
Let $E/K$ be a finite field extension. What is the smallest number $k$ such that there exist a tower of field extensions $K=K_0\subset K_1\subset \ldots \subset K_n$ with the property that $E$ is contained in $K_n$ and each extension $K_i/K_{i-1}$ is of essential dimension at most $k$?
\end{question}

In the language of branched coverings it would amount to the following:

\begin{question} \label{question:hilbert_branched_coverings}
Let $\xi_X$ be a branched covering over an irreducible variety $X$. What is the smallest number $k$ for which one can find a Zariski open set $U\subset X$ and a tower of branched coverings $X_n\to X_{n-1}\to \ldots \to X_0=X$ such that the restriction of $\xi_X$ to $U$ is a covering which is dominated by the restiction of the composite branched covering $X_n\to X_0$ to $U$ and such that each branched covering $X_i\to X_{i-1}$ can be rationally induced from a branched covering over a space of dimension $\leq k$?
\end{question}

While we can't say anything intelligent about this question, we can prove some lower bound on the length of the tower for any fixed $k$. To formulate a precise result we will need the following definition:

\begin{definition}
A branched covering $\xi_X$ on a variety $X$ is said to be \textbf{dominated by a tower of extensions of dimensions $k_1,\ldots,k_n$} if there exists a tower of branched coverings $X_n\to X_{n-1}\to \ldots \to X_0=X$ such that $\xi_X$ is a subcovering of a covering dominated by the covering $X_n\to X$ over some Zariski open set $U\subset X$ and each covering $X_i\to X_{i-1}$ can be rationally induced from a space of dimension at most $m_i$.
\end{definition}

\subsection{Especially Interesting Cases}
\label{section:interesting_cases}

Due to its universal nature the case when $X=\mathbf{C}^n$ and $\mathbf{z}=\mathbf{z}(x_1,\ldots,x_n)$ is the universal algebraic function satisfying $\mathbf{z}^n+x_1\mathbf{z}^{n-1}+\ldots+x_n=0$ was especially interesting to classics. This case was considered by Kronecker and Klein (for example in \cite{Klein1} for $n=5$).

Classics were also interested in the special case when the algebraic function is as before, but the domain on which it is defined supports the square root of the discriminant as a rational function on it. Namely $$X=\{(x_1,\ldots,x_n,D)|d^2=\text{ discriminant of } \mathbf{z}^n+x_1\mathbf{z}^{n-1}+\ldots+x_n=0\}$$ In particular Kronecker showed that for $n=5$ this function can't be rationally induced from a space of dimension one.


\section{Algebraic Functions - Definition}
\label{section:algebraic_functions}

Since the notion of ``algebraic function" on a variety $X$ can be a little ambiguous, in this section we provide the  definitions that will make question \ref{question:algebraic_functions} and its relation to question \ref{question:branched_coverings} clearer.

\begin{definition}
An algebraic function $z$ on an irreducible variety $X$ is a choice of a branched covering $\tilde{X}\to X$ and a regular function $z:\tilde{X}\to \mathbf{C}$. The variety $\tilde{X}$ is called a domain of definition of $z$.

An algebraic function $z'$ with domain $\tilde{X}'$ is called a restriction of algebraic function $z$ with domain $\tilde{X}$ if there exist a branched covering $\tilde{X}'\to \tilde{X}$ making the following diagram commutative

$$\xymatrix{ \tilde{X}' \ar[d] \ar[dr]^{w} & \\
\tilde{X} \ar[r]^z \ar[d] & \mathbf{C} \\
X & 
}$$

Two algebraic functions are called equivalent if they are both restrictions of the same algebraic function.
\end{definition}

An algebraic function has in fact a natural domain. Namely to a function 
$$\xymatrix{\tilde{X}\ar[r]^z \ar[d]^{p} & \mathbf{C} \\ X & }$$ we associate an equivalent algebraic function with domain $\tilde{X}_z=\{(x,t)\in X\times \mathbf{C}|\exists \tilde{x}\in\tilde{X} \mbox{with } p(\tilde{x})=x, z(\tilde{x})=t\}$. We then define $z(x,t)=t$ and $p(x,t)=x$ for $(x,t)\in \tilde{X}_Z$. We also define a map from $\tilde{X}$ to $\tilde{X}_Z$ by sending $\tilde{x}\in\tilde{X}$ to $(p(\tilde{x}),z(\tilde{x}))$. With these definitions the following diagram becomes commutative

$$\xymatrix{
\tilde{X} \ar[ddr]_p \ar[drr]^z
\ar[dr] \\
& \tilde{X}_z \ar[d]^p \ar[r]_z
& \mathbf{C} \\
& X  &  }$$
showing that the function we started with is equivalent to the one we defined.

Given two algebraic functions $z_1,z_2$ with domains $\tilde{X}_1$ and $\tilde{X}_2$ respectively, one can find a common domain for them (i.e. find $\tilde{X}$, a map $\tilde{X}\to X$ and functions $z_1',z_2'$ on $\tilde{X}$ such that $z_i'$ is a restriction of $z_i$). Namely one can take $\tilde{X}=\tilde{X}_1\times_X\tilde{X}_2$ and $z_i'$ to be the pullback of $z_i$ to $\tilde{X}$ through the obvious maps from $\tilde{X}$ to $\tilde{X}_1$ and $\tilde{X}_2$.

With this construction one can define sums, products and quotients of algebraic functions (the quotient being defined only where the denominator doesn't vanish). Thus the notion of composition of a rational function and algebraic functions, needed for question \ref{question:algebraic_functions} is defined as well.

\begin{remark}
Since the domain of an algebraic function is not assumed to be irreducible, the algebraic function might have several independent branches.
\end{remark}
\begin{remark}
According to our definition the sum $\sqrt{x}+\sqrt{x}$ is defined as $z+w$ on the variety  $\{(x,z,w)\in\mathbf{C}^3|z^2=x,w^2=x\}$, i.e. it has two independent branches: $2\sqrt{x}$ and $0$.
\end{remark}

\subsection{From Algebra to Topology}

The following lemma allows us to use topological considerations to approach question \ref{question:algebraic_functions}:

\begin{lemma} \label{lemma:from_alg_fun_to_covering}
Suppose that an algebraic function $\mathbf{z}$ over a variety $X$ is rationally induced from an algebraic function $\mathbf{w}$ on a variety $Y$ of dimension $k$. Then there exists a Zariski open subset $U$ such that the covering associated to the restriction of the algebraic function $\mathbf{z}$ to $U$ can be induced from a topological space of dimension $\leq k$.
\end{lemma}

\begin{proof}
According to definition \ref{definition:rationally_induced} there exist a dominant rational morphism $r$ from $X$ to $Y$, a rational function $R$ on $X\times\mathbf{C}$ and a Zariski open set $U$ of $X$ such that $\mathbf{z}(x)$ is a branch of the function $R(x,\mathbf{w}(r(x)))$ for $x\in U$.

By replacing $Y$ by its Zariski open subset and shrinking $U$ if necessary, we can assume that $Y$ is affine. By shrinking $U$ further we can also assume that the covering associated to the algebraic function $x\to R(x,\mathbf{w}(r(x)))$ is unramified over $U$.

Since $Y$ is affine variety, it is Stein and hence is homotopically equivalent to a topological space of dimension $\leq k$.  In particular the covering associated to the algebraic function $x\to \mathbf{w}(r(x))$ over $U$ can be induced from a space of dimension $\leq k$. Since the covering associated to $x\to R(x,\mathbf{w}(r(x)))$ is dominated by it, corollary \ref{corollary:dominated} implies that is also can be induced from a space of dimension $\leq k$. Finally, because the function $\mathbf{z}(x)$ is a branch of $x\to R(x,\mathbf{w}(r(x)))$, the covering associated to it can also be induced from a space of dimension $\leq k$. 
\end{proof}

In a similar fashion we can prove the following:

\begin{lemma} \label{lemma:from_alg_fun_to_covering_comp}
Suppose that an algebraic function $\mathbf{z}$ over a variety $X$ is dominated by a tower of extensions of dimensions $k_1,\ldots,k_n$. Then there exists a Zariski open subset $U$ such that the covering associated to the restriction of the algebraic function $z$ to $U$ is dominated by a covering that is a composition of coverings $U_n\to U_{n-1}\to \ldots \to U_0=U$ such that each $U_i\to U_{i-1}$ can be induced from a space of dimension at most $k_i$.
\end{lemma}

\section{Algebraic functions on the algebraic torus}
\label{section:algebraic_tori}

In this section we will answer completely questions \ref{question:algebraic_functions} for algebraic functions unramified on $(\mathbf{C}\smallsetminus \{0\})^n$. Before we do so, we show by example that the problem is not completely trivial.

\begin{example}
Let $\mathbf{z}(x,y)=\sqrt{x}+\sqrt[3]{y}$. We claim that it is induced from an algebraic function of one variable. Namely, one can verify that $$\sqrt{x}+\sqrt[3]{y}=\frac{y}{x}\left(\left(\sqrt[6]{\frac{x^3}{y^2}}\right)^2+\left(\sqrt[6]{\frac{x^3}{y^2}}\right)^3\right)$$ so if we let $R(x,y,w)=\frac{y}{x}(w^2+w^3)$, $\mathbf{w}(r)=\sqrt[6]{r}$, $r(x,y)=\frac{x^3}{y^2}$ then $\mathbf{z}(x,y)=R(x,y,\mathbf{w}(r(x,y)))$ for $x,y\neq0$.

On the other hand a similarly looking function $\sqrt{x}+\sqrt{y}$ can't be induced from an algebraic function of one variable, as theorem \ref{theorem:functions_on_torus} below shows.
\end{example}

Now we state the main result of this section. In what follows $\mathbf{T}$ stands for $\mathbf{C}\smallsetminus\{0\}$.

\begin{theorem}
\label{theorem:functions_on_torus}
Let $\mathbf{z}$ be an algebraic function on the torus $\mathbf{T}^n$ unramified over $\mathbf{T}^n$. Let $k$ denote the rank of its monodromy group. Then $\mathbf{z}$ can be induced from an algebraic function on $\mathbf{T}^k$ and it cannot be rationally induced from an algebraic function on a variety of dimension $<k$.

Moreover, it is dominated by a tower of extensions of dimensions $k_1,\ldots,k_s$ if and only if $k_1+\ldots+k_s\geq k$.
\end{theorem}

\begin{proof}
We first show that $\mathbf{z}$ can be rationally induced from an algebraic function on $\mathbf{T}^k$.

A choice of coordinates $x_1,\ldots,x_n$ on $\mathbf{T}^n$ gives rise to a choice of generators $\gamma_1,\ldots,\gamma_n$ of $\pi_1(\mathbf{T}^n)$, because $\mathbf{T}^n$ retracts to the torus $|x_1|=1,\ldots,|x_n|=1$ and the corresponding $\gamma_i$ is the loop in this torus for which all $x_j$ are constant for $j\neq i$ (with $dx_i/x_i$ defining the positive orientation on it). A toric change of coordinates in $\mathbf{T}^n$ gives rise to a linear change of generators in $\pi_1(\mathbf{T}^n)$.

Let $A$ denote the subgroup of loops in $\pi_1(\mathbf{T}^n)$ that leave all the branches of $\mathbf{z}$ invariant under the monodromy action.

Choose coordinates $x_1,\ldots,x_n$ in $\mathbf{T}^n$ so that $A=\langle \gamma_1^{m_1},\ldots,\gamma_n^{m_n}\rangle$ with $m_n|m_{n-1}|\ldots|m_1$ (this is possible because of Smith normal form theorem mentioned in the proof of lemma \ref{lemma:normal_form}). Since $k$ is the rank of the monodromy group $\pi_1(\mathbf{T}^n)/A$, we have $m_{k+1}=1,\ldots,m_n=1$.

The function $\psi(x_1,\ldots,x_n)=\mathbf{z}(x_1^{m_1},\ldots,x_n^{m_n})$ is invariant under monodromy action, hence is rational. Hence $$z(x_1,\ldots,x_n)=\psi(x_1^{1/m_1},\ldots,x_k^{1/m_k},x_{k+1},\ldots,x_{n})$$ where $\psi$ is rational.

Principal element theorem implies that the field extension $$\mathbf{C}(x_1^{1/m_1},\ldots,x_k^{1/m_k})/\mathbf{C}(x_1,\ldots,x_k) $$ is generated by one element, say the algebraic function $\mathbf{w}(x_1,\ldots,x_k)$. But then each $x_i^{1/m_i}$ is a rational function of $\mathbf{w}$: $x_i^{1/m_i}=r_i(x,\mathbf{w}(x))$, where $x$ stands for $(x_1,\ldots,x_k)$. 

Hence the function $$\mathbf{z}(x_1,\ldots,x_n)=\psi(r_1(\mathbf{w}(x),x),\ldots,r_k(\mathbf{w}(x),x),x_{k+1},\ldots,x_n)$$ where $x=(x_1,\ldots,x_k)$ is rationally induced from the function $\mathbf{w}$ on $\mathbf{T}^k$.

Moreover, if $k_1+\ldots+k_s\geq k$ then the function $\mathbf{z}$ lies in the extension of the field of rational functions on $\mathbf{T}^n$ by first adding to it the first $k_1$ functions $x_i^{m_i}$, then the next $k_2$ and so on. By what we have already showed $j$-th step can be accomplished by adding one algebraic function that can be rationally induced from a space of dimension $k_j$. This shows that $\mathbf{z}$ is dominated by a tower of extensions of dimensions $k_1,\ldots,k_s$.

Now suppose that the function $\mathbf{z}$ is rationally induced from an algebraic function over a variety $Y$ of dimension smaller than $k$. 

Lemma \ref{lemma:from_alg_fun_to_covering} then implies that there exists a Zariski open subset $U$ of $\mathbf{T}^n$ over which the covering associated to the algebraic function $\mathbf{z}$ can be induced from a topological space of dimension $< k$.

It follows from the results in \cite{Maz} that for sufficiently small $\epsilon_1,\ldots,\epsilon_n$ the torus $|x_1|=\epsilon_1,\ldots,|x_n|=\epsilon_n$ lies entirely inside $U$.

The space $\mathbf{T}^n$ can be retracted onto this torus. Hence the monodromy group of the restriction of the covering associated to $\mathbf{z}$ to this torus coincides with the full monodromy group of $\mathbf{z}$ over $\mathbf{T}^n$ and thus its rank is $k$ as well. But then theorem \ref{theorem:mainresult} tells that this covering can't be induced from a covering over a space of dimension $<k$. This however contradicts lemma \ref{lemma:from_alg_fun_to_covering}.

Similar proof shows that if the function $\mathbf{z}$ is dominated by a tower of extensions of dimensions $k_1,\ldots,k_s$ then $k\leq k_1+\ldots+k_s$, except instead of lemma \ref{lemma:from_alg_fun_to_covering} we use lemma \ref{lemma:from_alg_fun_to_covering_comp} and instead of the topological result \ref{theorem:mainresult} we use the theorem \ref{theorem:mainresult_comp}.
\end{proof}

We can use this theorem to prove some bounds for the questions in section \ref{section:interesting_cases}:

\begin{theorem}
\label{theorem:universal_function}
The universal algebraic function $\mathbf{z}(x_1,x_2,\ldots,x_n)$, i.e. the function satisfying $\mathbf{z}^n+x_1\mathbf{z}^{n-1}+\ldots+x_n=0$ can't be rationally induced from an algebraic function over a space of dimension $<\lfloor n/2\rfloor$.
\end{theorem}
\begin{proof}
Suppose that $n=2k$ is even. 





Consider the mapping that sends $(a_1,\ldots,a_k,s_1,\ldots,s_k)\in \mathbf{C}^{2k}$ to the coefficients $(x_1,\ldots,x_n)$ satisfying $$\prod_{i=1}^k (\mathbf{w}-(a_i+\sqrt{s_i}))(\mathbf{w}-(a_i-\sqrt{s_i}))=\mathbf{w}^n+x_1 \mathbf{w}^{n-1}+\ldots+x_n$$ It is easy to check that this mapping is onto, hence if the function $\mathbf{z}$ can be rationally induced from an algebraic function over a space of dimension at most $k$, then the pullback of $\mathbf{z}$ through this mapping also can.

Notice however that the pullback of $\mathbf{z}$ through this mapping is the function $\mathbf{w}=\mathbf{w}(a_1,\ldots,a_k,s_1,\ldots,s_k)$ satisfying $$\prod_{i=1}^k \left(\mathbf{w}-(a_i+\sqrt{s_i})\right)\left(\mathbf{w}-(a_i-\sqrt{s_i})\right)=0$$ This function is an algebraic function unramified over the algebraic torus with coordinates $a_1,\ldots,a_k,s_1,\ldots,s_k$ and its monodromy group is isomorphic to $\mathbf{Z}_2^k$, i.e. has rank $k$. This contradicts theorem \ref{theorem:functions_on_torus}.

If $n=2k+1$ is odd, we can apply the same argument to the function $\mathbf{w}$ satisfying $$\left(\prod_{i=1}^k \left(\mathbf{w}-(a_i+\sqrt{s_i})\right)\left(\mathbf{w}-(a_i-\sqrt{s_i})\right)\right)(\mathbf{w}-a_{k+1})=0$$
\end{proof}

Similar technique can be applied to analyze what happens if the square root of discriminant is adjoined to the domain of rationality. Namely we can prove the following theorem:

\begin{theorem}
\label{theorem:universal_function_with_discriminant}
The algebraic function $\mathbf{z}(x_1,x_2,\ldots,x_n)$ the function satisfying $z^n+x_1z^{n-1}+\ldots+x_n=0$ considered as a function on the variety $$\{(x_1,\ldots,x_n,d)\in\mathbf{C}^n\times\mathbf{C}|d^2=\text{discriminant of }z^n+x_1z^{n-1}+\ldots+x_n=0\}$$ can't be rationally induced from an algebraic function over a space of dimension $<2\lfloor n/4\rfloor$.
\end{theorem}

\begin{proof}
Suppose that $n=4k$ is divisible by four.

Let $w_{\cdot}$ denote the expressions
\begin{align*}
 w_{4i}\hspace{1pc} &= a_i+\sqrt{s_i}+\sqrt{t_i}+b_i\sqrt{s_i t_i} \\
 w_{4i+1} &= a_i-\sqrt{s_i}+\sqrt{t_i}-b_i\sqrt{s_i t_i} \\
 w_{4i+2} &= a_i+\sqrt{s_i}-\sqrt{t_i}-b_i\sqrt{s_i t_i} \\
 w_{4i+3} &= a_i-\sqrt{s_i}-\sqrt{t_i}+b_i\sqrt{s_i t_i}
\end{align*}
 and let the function $\mathbf{w}=\mathbf{w}(a_1,\ldots,a_m,b_1,\ldots,b_m,s_1,\ldots,s_m,t_1,\ldots,t_m)$ satisfy $\prod_{i=1}^{4m}{(w-w_i)}=0$. The monodromy of this algebraic function is realized by even permutations only, hence its discriminant is a square of some rational function in the variables $a,b,s,t$. Hence this algebraic function can be induced from the function $\mathbf{z}$ in the formulation of the theorem by means of a map that sends the point  $(a_1,\ldots,a_m,b_1,\ldots,b_m,s_1,\ldots,s_m,t_1,\ldots,t_m)$ to the point $(x_1,\ldots,x_n,d)$ where $x_1,\ldots,x_n$ are the coefficients of the expanded version of the equation $\prod_{i=1}^{4m}{(z-w_i)}=z^n+x_1z^{n-1}+\ldots+x_n$ that $\mathbf{w}$ satisfies and $d$ is the rational function whose square is equal to the discriminant of $\mathbf{w}$. The image of this mapping is in fact dense in $X$. Indeed, if $(x_1,\ldots,x_{4m},d)$ is a point in $X$, denote by $z_1,\ldots,z_{4m}$ the roots of the equation $z^n+x_1z^{n-1}+\ldots+x_n=0$. Then the equations
 \begin{align*}
 a_i &= \frac{z_{4i}+z_{4i+1}+z_{4i+2}+z_{4i+3}}{4} \\
 \sqrt s_i &= \frac{z_{4i}-z_{4i+1}+z_{4i+2}-z_{4i+3}}{4} \\
 \sqrt t_i &= \frac{z_{4i}+z_{4i+1}-z_{4i+2}-z_{4i+3}}{4} \\
 b_i \sqrt {s_i t_i} &= \frac{z_{4i}-z_{4i+1}-z_{4i+2}+z_{4i+3}}{4} 
\end{align*}
are clearly solvable for the variables $a_i,b_i,s_i,t_i$ for $(z_1,\ldots,z_{4m})$ in a Zariski open subset of $\mathbf{C}^{4m}$ and the solution is a point that gets mapped either to $(x_1,\ldots,x_{4m},d)$ or to $(x_1,\ldots,x_{4m},-d)$. Since $X$ is irreducible, the image is a dense Zariski open set in $X$.

Hence if the function $\mathbf{z}$ can be rationally induced from an algebraic function on a space of dimension $<2m$, the function $\mathbf{w}$ also can. However $\mathbf{w}$ is an algebraic function that is unramified on the algebraic torus with coordinates $a,b,s,t$ and its monodromy group is isomorphic to $(\mathbf{Z}_2^2)^m$, i.e. has rank $2m$. By theorem \ref{theorem:functions_on_torus} this function can't be rationally induced from an algebraic function on a space of dimension $<2m$.

In case $n=4m+1$ we can consider instead of $\mathbf{w}$ from the argument above the function $\mathbf{w}(a_1,\ldots,a_{m+1},b_1,\ldots,b_m,s_1,\ldots,s_m,t_1,\ldots,t_m)$ satisfying $\left(\prod_{i=1}^{4m}{(w-w_i)}\right)(w-a_{m+1})=0$ with the same $w_i$ as above.

Cases $n=4m+2$ and $n=4m+3$ can be handled in the same way.
\end{proof}

\section{Local version}
\label{section:local_version}

Theorem \ref{theorem:functions_on_torus} about algebraic functions unramified on the algebraic torus $(\mathbf C^*)^n$ has a local analogue:

\begin{theorem}
 \label{theorem:functions_on_torus_local}
Let $\mathbf{z}$ be a germ at the origin $(0,\ldots,0)$ of an algebraic function defined on $(\mathbf{C}\smallsetminus \{0\})^n$ such that for every algebraic function representing the germ there exists an $\epsilon>0$ such that this algebraic function is unramified  on the punctured polydisc $\{(x_1,\ldots,x_n)\in(\mathbf{C}^*)^n\text{ with }0<|x_i|<\epsilon\text{ for all }i \}$. Let $k$ denote the rank of its monodromy group on this punctured polydisc (it is obviously the same for all representatives of the germ). Then the restriction of $\mathbf{z}$ to this polydisc can be rationally induced from an algebraic function on $\mathbf{C}^k$ and it cannot be rationally induced from an algebraic function on a variety of dimension $<k$ by means of a germ at the origin of a rational mapping.

Moreover $\mathbf{z}$ is dominated by a germ at origin of a tower of extensions of dimensions $k_1,\ldots,k_s$ if and only if $k_1+\ldots+k_s\geq k$.
\end{theorem}

The proof of this version of theorem \ref{theorem:functions_on_torus} practically coincides with the proof of theorem \ref{theorem:functions_on_torus} itself.

We will now present a construction that allows to use this result to obtain some information about any algebraic function. To do so we recall the concept of a Parshin point and a neighbourhood of a Parshin point:

Let $X$ be a variety. Let $V_\cdot$ be a flag of germs at a point $p\in X$ of varieties $V_n\supset V_{n-1} \supset \ldots \supset V_0=\{p\}$ with $\dim{V_i}=i$ and each $V_{i}$ irreducible along $V_{i-1}$. Such a flag will be referred to as a \textbf{Parshin point} of $X$.

 In \cite{Maz1} Mazin shows that for any Parshin point and any Zariski open set $U\subset X$ one can find a regular mapping $\phi: D\to X$ from a polydisc $D=\{(x_1,\ldots,x_n)\in\mathbf{C}^n$ with $|x_i|<\epsilon$ for all $i$ $\}$ to $X$ sending the standard flag in $\mathbf{C}^n$ (i.e. $\mathbf{C}^n\supset Z(x_1) \supset \ldots \supset Z(x_1,\ldots,x_n)$, where $Z(x_1,\ldots,x_k)$ denotes the germ at origin of the set where $x_1=\ldots=x_k=0$) to the flag $V_\cdot$ and sending the compliment to the coordinate cross isomorphically onto a (classically) open subset of $U$ contained in the compliment to $V_{n-1}$.
 
Let now $\mathbf{z}$ be an algebraic function defined on a Zariski open subset $U$ of a variety $X$. By shrinking $U$ we can assume that this function is unramified over $U$. Then any Parshin point in $X$ gives rise to a neighbourhood in the above sense and hence, via pullback, to an algebraic function $\phi^*\mathbf{z}$ on the polydisc $D$ unramified over the punctured polydisc $\{(x_1,\ldots,x_n)\in(\mathbf{C}^*)^n$ with $0<|x_i|<\epsilon$ for all $i\}$. To this function we can apply theorem \ref{theorem:functions_on_torus_local}: if the monodromy group of $\phi^*\mathbf{z}$ on the punctured polydisc has rank $k$, then the original function $\mathbf{z}$ can't be rationally induced from an algebraic function  on a variety of dimension $<k$.

Thus we arrive at the following theorem:

\begin{theorem}
\label{theorem:local_obstruction}
Let $z$ be an algebraic function defined on a variety $X$. Suppose that there exists a Parshin point on $X$ and its punctured neighbourhood as described above such that the monodromy group of $z$ on this punctured neighbourhood has rank $k$. Then $z$ can't be rationally induced from an algebraic function on a variety of dimension $<k$.

Moreover $z$ is not dominated by a tower of extensions of dimensions $k_1,\ldots,k_s$ if  $k_1+\ldots+k_s < k$.
\end{theorem}

We can apply theorem \ref{theorem:local_obstruction} above to reprove theorems \ref{theorem:universal_function} and \ref{theorem:universal_function_with_discriminant}.

All we have to do is exhibit flags around which the monodromy of the universal algebraic function of order $n$ has rank $\lfloor n/2\rfloor$ (for theorem \ref{theorem:universal_function}) or has rank $2\lfloor n/4\rfloor$ and is realized by even permutations only (for theorem \ref{theorem:universal_function_with_discriminant}).

We will show how to choose such flags for the case $n=4$.

Let $\mathbf{z}$ be the function satisfying $z^4+x_1z^3+x_2z^2+x_3z+x_4=0$. Let $p:\mathbf{C}^4\to \mathbf C^4$ be the function sending roots $z_1,\ldots,z_4$ of this equation to its coefficients $x_1,\ldots,x_4$.

The branched covering $p$ is the Galois covering associated to $\mathbf{z}$, so it's enough to consider it instead of the function $\mathbf{z}$.

We will exhibit the flags we are interested in as images under $p$ of flags in $\mathbf{C}^4$ with coordinates $z_1,\ldots,z_4$.

The first flag is the image of the flag $Z(z_1=z_2)\supset Z(z_1=z_2,z_3=z_4) \supset Z(z_1=z_2=z_3=z_4) \supset Z(z_1=z_2=z_3=z_4=0)$ where $Z(\mbox{equation})$ stands for the set of points $(z_1,\ldots,z_4)$ for which the equation holds. This is a flag of irreducible varieties and its image under $p$ is a Parshin point.

The branched covering $p$ realizes the quotient of $\mathbf{C}^4$ by the action of the permutation group $S_4$ acting by permuting coordinates. Hence the monodromy of $p$ in a neighbourhood of flag can be identified with the group of all permutations that stabilize any of the flags in its preimage. In our case the permutations that stabilize the flag $Z(z_1=z_2)\supset Z(z_1=z_2,z_3=z_4) \supset Z(z_1=z_2=z_3=z_4) \supset Z(z_1=z_2=z_3=z_4=0)$ are the trivial permutation, $(z_1,z_2)$, $(z_3,z_4)$ and $(z_1,z_2)(z_3,z_4)$. Hence the monodromy of $\mathbf{z}$ around the flag has rank 2.

\input{./cycle.tex}

For the situation where we are only allowing even permutations, the flag $Z(z_1+z_2=z_3+z_4)\supset Z(z_1=z_3,z_2=z_4) \supset Z(z_1=z_2=z_3=z_4) \supset Z(z_1=z_2=z_3=z_4=0)$ has the desired properties. The permutations of $S_4$ that stabilize this flag are the trivial permutation, $(z_1,z_2)(z_3,z_4)$, $(z_1,z_3)(z_2,z_4)$ and their product $(z_1,z_4)(z_2,z_3)$. These permutations are even. Hence the monodromy of $\mathbf{z}$ around the image of this flag under $p$ is of rank $2$ and consists of even permutations only.

For larger values of $n$ the corresponding flags should be the images of
\begin{itemize}
\item $$Z(z_1=z_2)\supset Z(z_1=z_2,z_3=z_4)\supset Z(z_1=z_2,z_3=z_4,z_5=z_6)\supset\ldots$$
continued  until there are no more unused pairs of coordinates to equate and then continued all the way to a point in an arbitrary manner.
\item \begin{align*} &Z(z_1+z_2=z_3+z_4)\supset Z(z_1=z_3,z_2=z_4) \supset \\
&Z(z_1=z_3,z_2=z_4, z_5+z_6=z_7+z_8) \supset Z(z_1=z_3,z_2=z_4, z_5=z_7,z_6=z_8) \supset \ldots 
\end{align*}
continued  until there are no more unused quadruples of coordinates to use and then continued all the way to a point in an arbitrary manner.
\end{itemize}

\section{Generic algebraic function of \texorpdfstring{$k$}{k} parameters and degree \texorpdfstring{$\geq 2k$}{at least 2k} can't be simplified}
\label{section:generic}

The same flag that we used to show that the universal function of $n$ parameters can't be reduced to less than $\lfloor n/2\rfloor$ parameters gives some useful information about any other algebraic function as well. Indeed, any algebraic function of degree $n$ can be induced from the universal algebraic function. Thus we can think of it as the restriction of the universal algebraic function to some subvariety $X$ in $\mathbf C^n$. To this function we can apply our arguments with the flag obtained by intersecting a flag in $\mathbf C^n$ with $X$.

\begin{notation}
Let $p:\mathbf{C}^n\to \mathbf{C}^n$ be the mapping that sends the point $(z_1,\ldots,z_n)$ to the coefficients $(x_1,\ldots,x_n)$ of the equation $z^n+x_1z^{n-1}+\ldots+x_n=0$. We will denote by $D_k$ the image under $p$ of the set $\{(z_1,\ldots,z_n)|z_1=z_2,\ldots,z_{2k-1}=z_{2k},\mbox{no other equalities hold between the } z_i\mbox{'s}\}$.
\end{notation}

Let $x^\circ$ be the image under $p$ of the point $(z_1^\circ,\ldots,z_n^\circ)$ with $z_1^\circ=z_2^\circ,\ldots,z_{2k-1}^\circ=z_{2k}^\circ$. The branches of the algebraic functions $$(z_1-z_2)^2,\ldots,(z_{2k-1}-z_{2k})^2,\frac{z_1+z_2}{2}-z_1^\circ,\ldots,\frac{z_{2k-1}+z_{2k}}{2}-z_{2k-1}^\circ,z_{2k+1}-z_{2k-1}^\circ,\ldots,z_{n}-z_{n}^\circ $$ that assume the value $0$ at the point $(x_1^\circ,\ldots,x_n^\circ)$ form a coordinate system in a neighbourhood of $x^\circ$. If we denote these coordinate functions by $\tilde{x}_1,\ldots,\tilde{x}_n$, then the discriminant is defined in a small neighbourhood of the point $z^\circ$ by the equation $\tilde{x}_1\cdot \ldots \cdot \tilde{x}_k=0$. This shows that $D_k$ is contained in the locus of the points where the discriminant variety has a normal crossing singularity. If we choose the coordinates on the source $\mathbf{C}^n$ to be $(\tilde{z}_1,\ldots,\tilde{z}_n)=(z_1-z_2,\ldots,z_{2k-1}-z_{2k},\frac{z_1+z_2}{2}-z_1^\circ,\ldots,\frac{z_{2k-1}+z_{2k}}{2}-z_{2k-1}^\circ,z_{2k+1}-z_{2k-1}^\circ,\ldots,z_{n}-z_{n}^\circ)$ then in these coordinates $p$ is given by the formula $(\tilde{x}_1,\ldots,\tilde{x}_n)=p(\tilde{z}_1,\ldots,\tilde{z}_n)=(\tilde{z}_1^2,\ldots,\tilde{z}_k^2,\tilde{z}_{k+1},\ldots,\tilde{z}_n)$.

\begin{theorem}
\label{theorem:transversalintersection}
 Let $\xi_X$ be the algebraic function obtained by restricting the universal algebraic function on $\mathbf C^n$ to a subvariety $X\subset \mathbf C^n$. Let $D_k$ be the subset of $\mathbf{C}^n$ defined above. Suppose that $X$ and $D_k$ intersect transversally at least at one point. Then the function $\xi_X$ can't be rationally induced from a function of less than $k$ parameters.
\end{theorem}

\begin{proof}
Let $p$ denote as before the function that sends the roots $z_1,\ldots,z_n$ of the equation $z^n+x_1 z^{n-1}+\ldots+x_n=0$ to its coefficients $x_1,\ldots,x_n$.

Let $x^\circ=p(z^\circ)$ be a point in the intersection of $X$ with $D_k$ and $z^\circ=(z_1^\circ,\ldots,z_n^\circ)$ is such that $z_1=z_2,\ldots,z_{2k-1}=z_{2k}$ and no other equalities hold between the $z_i$'s. As we noted before, one can find coordinates $\tilde{x}_1,\ldots,\tilde{x}_n$ and $\tilde{z}_1,\ldots,\tilde{z}_n$ in small neighbourhoods of the points $x^\circ$ and $z^\circ$ respectively so that the mapping $p$ in these coordinates is simply $p(\tilde{z}_1,\ldots,\tilde{z}_n)=(\tilde{z}_1^2,\ldots,\tilde{z}_k^2,\tilde{z}_{k+1},\ldots,\tilde{z}_n)$.

Since $D_k$ is given in these coordinates by the equations $\tilde{x}_1=\ldots=\tilde{x}_k=0$ and $X$ is transversal to it, one sees that $X$ is transversal to the map $p$ and hence its preimage $Z=p^{-1}(X)$ is a manifold in a neighbourhood of the point $z^\circ$. The formula for $p$ shows that the tangent space to $Z$ at the point $z^\circ$ contains the vectors $\partial/\partial \tilde{z}_1,\ldots,\partial/\partial \tilde{z}_k$. In particular the differentials $d \tilde{z}_1,\ldots,d \tilde{z}_k$ are linearly independent on this tangent space and hence the functions $\tilde{z}_1,\ldots,\tilde{z}_k$ can be extended to local coordinates on $Z$ by adding if necessary some of the other $\tilde{z}_i$'s. Let's assume without loss of generality that $\tilde{z}_1,\ldots,\tilde{z}_k,\ldots,\tilde{z}_m$ are local coordinates on $Z$. By the transversality condition $X\pitchfork D_k$ and the definition of $Z$ as the preimage of $X$ it follows that $\tilde{x}_1,\ldots,\tilde{x}_m$ are local coordinates on $X$ at $x^\circ$. In these coordinates the restriction of $p$ on $Z$ is given by $p(\tilde{z}_1,\ldots,\tilde{z}_m)=(\tilde{z}_1^2,\ldots,\tilde{z}_k^2,\tilde{z}_{k+1},\ldots,\tilde{z}_m)$. Hence the local monodromy of $\xi_X$ around the flag on $X$ given by $\{\tilde{x}_1=0\}\supset\ldots\supset\{\tilde{x}_1=\ldots=\tilde{x}_m=0\}$ is of rank $k$ and hence $\xi_X$ is not rationally induced from any algebraic function with less than $k$ parameters.
 
\end{proof}

Much weaker assumptions than transversality of intersection are in fact needed for the conclusions of the theorem to hold. We won't need this greater generality however.

We will use this theorem now to show that a generic algebraic function depending on $k$ parameters and having degree $n$ can't be rationally induced from an algebraic function of less than $k$ parameters provided that $n\geq 2k$. The word ``generic'' can be made precise in many ways. What follows is one of them:

\begin{theorem}
Let $L$ be a linear space of polynomials on $\mathbf C^n$ that contains constants and linear functions. Then for generic $p_1,\ldots,p_{n-k}\in L$ the algebraic function obtained from the universal algebraic function on $\mathbf C^n$ by restriction to the set $p_1(x)=0,\ldots,p_{n-k}(x)=0$ can't be rationally induced from an algebraic function of less than $k$ parameters provided that $n\geq 2k$.
\end{theorem}

\begin{proof}
 According to theorem \ref{theorem:transversalintersection} it is enough to show that for generic $p_1,\ldots,p_{n-k}\in L$ the intersection of $D_k$ with  $p_1(x)=0,\ldots,p_{n-k}(x)=0$ is non-empty and transversal. Since the set of polynomials for which the intersection is non-empty and transversal is an algebraic set, it is enough to show it has non-zero measure.
 
 The fact that the set of equations in $L^k$ whose zero set intersects $D^k$ transversally is of full measure follows from Sard's lemma.
 
 Indeed, consider the subset $I=\{(f,x)|f(x)=0,f\in L^k, x\in \mathbf C^n\}$. This subset is a submanifold of $L^{n-k}\times \mathbf{C}^n$. Indeed, the differential of the evaluation function $(f,x)\to \mathbf C^k$ evaluated on a tangent vector $(\phi,\xi)\in T_{(f,x)}L^{n-k}\times \mathbf C^n$ is equal to $\phi(x)+d_x f(\xi)$. Since $L$ contains all constants, this differential is of full rank at all points.
 
 For every point $f\in L^{n-k}$ which is regular for the projection from $I$ to $L^{n-k}$, the zero set of $f$ is a submanifold of $\mathbf C^n$, hence by Sard's lemma the zero set of $f$ is a submanifold of $\mathbf C^n$ for almost all $f\in L^{n-k}$.
 
 In a similar way we can show it is transversal to $D_k$ for almost all $f\in L^{n-k}$.

 Indeed, the projection from $I$ onto $\mathbf C^n$ is a submersion (if we fix $\xi\in T_x \mathbf C^n$, we can choose $\phi\in T_f L^{n-k}$ so that $\phi(x)+d_x f(\xi)=0$, because $L$ contains all constants).
 
 Hence the preimage $I_D$ of $D_k$ under this projection is a submanifold of $I$.
 
 Now we claim that for any point $f\in L^{n-k}$ which is regular for the projection from $I_D$ to $L^{n-k}$ the set $\{x\in \mathbf C^k| f(x)=0\}$ is transversal to $D_k$. Indeed, let $x\in D_k$ be a point such that $f(x)=0$ and let $\xi$ be any vector in $T_x \mathbf C^n$. As we noted before the projection from $I$ to $\mathbf C^n$ is a submersion and hence we can find $\phi \in T_f L^{n-k}$ such that $(\phi,\xi)\in T_{(f,x)} I$, i.e. $d_x f(\xi)=-\phi(x)$. Since $f$ is a regular point of the projection from $I_D$ to $L^{n-k}$, one can find a vector $\xi'\in T_x D_k$ such that $(\phi,\xi')\in T_{(f,x)}I_D$. Thus $d_x f(\xi)=-\phi(x)=d_x f(\xi')$, i.e. $\xi-\xi'\in \ker d_x f$. Hence $\xi$ is a sum of a vector tangent to $D_k$ and a vector tangent to the level set of $f$, i.e. the level set of $f$ is indeed transversal to $D_k$.
 
 Sard's lemma then guarantees that for almost any $f$ the level set of $f$ is transversal to $D_k$.
 
Finally we have to show that the set of equations having at least one solution on $D_k$ is of full measure.

Suppose that the dimension of the space $L$ is equal to $l$.

 All fibers of the projection from $I_D$ to $\mathbf C^n$ are $(n-k)(l-1)$-dimensional. Indeed, the condition that an equation in $L$ vanishes at a given point $x$ is a linear condition and it is never satisfied by all equations in $L$ as $L$ contains constants. Thus the space $I_D$ is $(n-k)l$-dimensional. Since the set $I_D$ is an affine manifold its image under projection to $L^{n-k}$ is either of full measure or is contained in a proper affine  subvariety of $L ^{n-k}$. Suppose that the latter is the case. Then the dimension of each component of the preimage of any point $f\in L ^ {n-k}$ is at least 1 by (\cite{Mum}, theorem 3.13). One can however find equations $f\in L^{n-k}$ whose zero set on $D_k$ contains an isolated point $x\in D_k$: for instance one can take affine functions that vanish at $x$ and whose differentials are linearly independent when restricted to the tangent space $T_x D_k$.
 
\end{proof}

We'll now give another result of a similar nature.

\begin{theorem}
Let $L$ be a linear space of polynomials on $\mathbf C^k$ that contains constants and linear functions. Then for generic $p_1,\ldots,p_{n}\in L$ the algebraic function satisfying $z^n + p_1(x) z^{n-1}+\ldots+p_n(x)=0, x\in \mathbf C ^k$ can't be rationally induced from an algebraic function of less than $k$ parameters provided that $n\geq 2k$.
\end{theorem}
\begin{proof}
The proof is similar to the proof of the previous theorem. We will show that for generic choice of $(p_1,\ldots,p_n)\in L^n$ the image of $\mathbf C^k$ under the map $(p_1,\ldots,p_n)$ is transversal to $D_k$ and intersects $D_k$ non-trivially.

Transversality follows from transversality theorem \cite{GP}: the evaluation map from $L^n\times \mathbf C^k$ is transversal to $D_k$ because $L$ contains all linear functions. Hence for $p$ in a set of full measure in $L^n$ the map $p:\mathbf C^k \to \mathbf C^n$ is transversal to $D_k$.

To show that for generic $p$ the image of this map intersects $D_k$ non-trivially, we notice that the evaluation map described above is onto (since $L$ contains all constants) and hence the preimage of $D_k$ is at least $ln$-dimensional. Hence if it's projection onto $L^n$ is not of full measure, it is contained in a proper subvariety of $L^n$. This implies then that for generic $p\in L^n$ all components of the intersection of $p(\mathbf C^k)$ with $D_k$ are at least one-dimensional. This however contradicts the fact that all constants are in $L$: using them we can send $\mathbf{C}^k$ to only one point in $D_k$.

\end{proof}

\begin{remark}
The condition that $L$ contains constants and linear functions can be somewhat weakened. It is in fact enough to require that $L$ contains at least some polynomials $p_1,\ldots,p_{n-k}$ such that the set $p_1(x)=0,\ldots,p_{n-m}(x)=0$ intersects $D_k$ and at least some polynomial in $L$ doesn't vanish at one of the points of intersection.
\end{remark}
\begin{remark}
Versions of the previous two theorems where being rationally induced from an algebraic function on a variety of dimension $<k$ is replaced with being dominated by a tower of extensions of dimensions $k_1,\ldots,k_s$ with $k_1+\ldots+k_s<k$ are also correct for the same reasons.
\end{remark}


\bibliography{klein}
\end{document}

%% file: cycle.tex
\begin{figure}[ht]
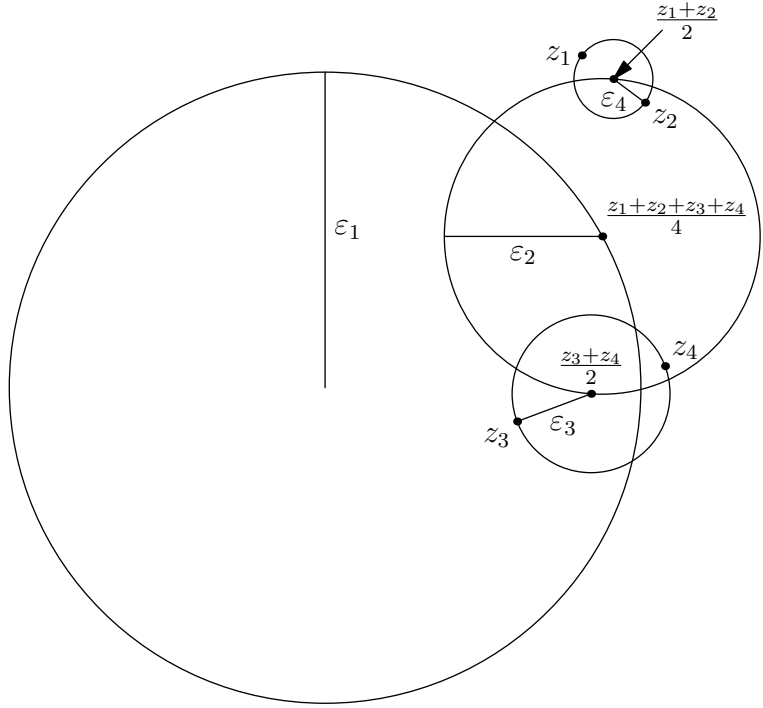

\centering
\begin{asy}
size (10cm);

pair x1,x2,x3,x4;

real angle1=0.5,angle2=1.5,angle3=2.5,angle4=3.5;

x1=16*(cos(angle1),sin(angle1))+8*(cos(angle2),sin(angle2))+2*(cos(angle3),sin(angle3));
x2=16*(cos(angle1),sin(angle1))+8*(cos(angle2),sin(angle2))-2*(cos(angle3),sin(angle3));
x3=16*(cos(angle1),sin(angle1))-8*(cos(angle2),sin(angle2))+4*(cos(angle4),sin(angle4));
x4=16*(cos(angle1),sin(angle1))-8*(cos(angle2),sin(angle2))-4*(cos(angle4),sin(angle4));

draw(circle((0,0),16));
draw((0,0)--(0,16));
label("$\varepsilon_1$",((0,0)+(0,16))/2,E);

label("$\frac{z_1+z_2+z_3+z_4}{4}$",(x1+x2+x3+x4)/4,0.3*NE);
dot((x1+x2+x3+x4)/4);

draw(circle((x1+x2+x3+x4)/4,8));
draw((x1+x2+x3+x4)/4--(x1+x2+x3+x4)/4+(-8,0));
label("$\varepsilon_2$",((x1+x2+x3+x4)/4+(x1+x2+x3+x4)/4+(-8,0))/2,S);

label("$\frac{z_1+z_2}{2}$",(x1+x2)/2,6*NE);
dot((x1+x2)/2);
draw((x1+x2)/2--((x1+x2)/2+3.5*NE),BeginArrow);

label("$\frac{z_3+z_4}{2}$",(x3+x4)/2,N);
dot((x3+x4)/2);

draw(circle((x3+x4)/2,4));
draw((x3+x4)/2--x3);
label("$\varepsilon_3$",((x3+x4)/2+x3)/2,SSE);

label("$z_3$",x3,SW);
dot(x3);

label("$z_4$",x4,NE);
dot(x4);

draw(circle((x1+x2)/2,2));
draw((x1+x2)/2--x2);
label("$\varepsilon_4$",((x1+x2)/2+x2)/2,0.2*SW);

label("$z_1$",x1,W);
dot(x1);

label("$z_2$",x2,SE);
dot(x2);

\end{asy}
\label{fig:cycle}
\caption{
In the case $n=4$ the torus on which the obstruction lives is the set of $x_1,\ldots,x_4$ corresponding to the roots $z_1,\ldots,z_n$ satisfying 
$|z_1-z_2|=2\varepsilon_4,|z_3-z_4|=2\varepsilon_3,|\frac{z_1+z_2}{2}-\frac{z_3+z_4}{2}|=2\varepsilon_2,|\frac{z_1+z_2+z_3+z_4}{4}|=\varepsilon_1$}

\end{figure}